\documentclass{amsart}
\usepackage{amsmath}

\usepackage{amssymb}
\usepackage{amsthm}
\usepackage{amscd,amsbsy}
\usepackage{tikz-cd}
\usetikzlibrary{arrows, matrix}
\usepackage{xcolor}
\usepackage{float}
\usepackage{mathtools}

\newtheorem{theorem}{Theorem}[section]
\newtheorem{lemma}[theorem]{Lemma}

\theoremstyle{definition}

\newtheorem{example}[theorem]{Example}

\newtheorem{definitions and remarks}[theorem]{Definitions and Remarks}
\newtheorem{question}[theorem]{Question}

\theoremstyle{remark}

\newtheorem{remark}[theorem]{Remark}

\numberwithin{equation}{section}

\newcommand{\bdry}{\mathrm{bdry}\,}
\newcommand{\interior}{\mathrm{int}\,}
\newcommand{\Sing}{\mathrm{Sing}\,}

\newcommand{\id}{\mathrm{id}}

\newcommand{\de}{{\delta}}

\newcommand{\ga}{{\gamma}}
\newcommand{\Ga}{{\Gamma}}

\newcommand{\la}{{\lambda}}

\newcommand{\s}{{\sigma}}

\newcommand{\vp}{{\varphi}}
\newcommand{\io}{{\iota}}

\newcommand{\IR}{{\mathbb R}}

\newcommand{\IZ}{{\mathbb Z}}

\newcommand{\cC}{{\mathcal C}}

\newcommand{\cI}{{\mathcal I}}

\newcommand{\cO}{{\mathcal O}}
\newcommand{\cP}{{\mathcal P}}

\newcommand{\wC}{{\widehat C}}
\newcommand{\wQ}{{\widehat Q}}

\newcommand{\RN}[1]{%
  \textup{\uppercase\expandafter{\romannumeral#1}}%
}

\begin{document}
\title[Global smoothing of a subanalytic set]
{Global smoothing of a subanalytic set}

\author[E.~Bierstone]{Edward Bierstone}
\address{University of Toronto, Department of Mathematics, 40 St. George Street,
Toronto, ON, Canada M5S 2E4}
\email[]{bierston@math.toronto.edu}
\author[A.~Parusi\'nski]{Adam Parusi\'nski}
\address{Universit\'e Nice Sophia Antipolis, 
CNRS, LJAD, UMR 7351, 06108 Nice, France}
\email[]{adam.parusinski@unice.fr}
\thanks{Research supported in part by NSERC grant OGP0009070.}



\begin{abstract}
We give rather simple answers to two long-standing
questions in real-analytic geometry, on global smoothing of a subanalytic set,
and on transformation of a proper real-analytic
mapping to a mapping with equidimensional fibres by global blowings-up of the target.
These questions are related: a positive answer to the second can be used to reduce
the first to the simpler semianalytic case. We show that the second question has a
negative answer, in general, and that the first problem nevertheless has a positive solution.
\end{abstract}

\date{\today}
\maketitle
\setcounter{tocdepth}{1}
\tableofcontents

\section{Introduction}\label{sec:intro}  
Semialgebraic and subanalytic sets have become ubiquitous in mathematics since
their introduction by {\L}ojasiewicz in the 1960s \cite{Loj}, following the celebrated Tarski-Seidenberg
theorem on quantifier elimination. In this article, we give rather simple answers to two
long-standing questions in real-analytic geometry, on global smoothing of a subanalytic set
(an analogue of resolution of singularities), and on transformation of a proper real-analytic
mapping to a mapping with locally equidimensional fibres by global blowings-up of the target
(a classical result of Hironaka in the complex-analytic case \cite{HiroAMJ}).

These questions are related: a positive answer to the second can be used to reduce
the first to the simpler semianalytic case. We show that the second question has a
negative answer, in general, and that the first problem nevertheless has a positive solution.
We are grateful to Masaki Kashiwara for his inquiries and suggestions about the global
smoothing problem.

\subsection{Global smoothing}\label{subsec:smoothing} Throughout the article, all spaces
and mappings are assumed to be defined over the field of real numbers, unless stated otherwise.
The results stated in this subsection will be proved in Section 2 below.

\begin{theorem}[Non-embedded global smoothing]\label{thm:nonemb}
Let $V$ be an analytic manifold of dimenson $n$, and let $X$ denote a closed subanalytic
subset of $V$, $\dim X =k$. Then there is a proper analytic mapping $\vp: X' \to V$, where
$X'$ is an analytic manifold of pure dimension $k$, and a smooth open subanalytic subset 
$U$ of $X$, where $\dim X{\backslash}U < k$, such that:
\begin{enumerate}
\item $\vp(X') \subset X$;
\item $\vp^{-1}(X{\backslash}U)$ is a simple normal crossings hypersurface $B' \subset X'$;
\item for each connected component $W$ of $U$, $\vp^{-1}(W)$ is a finite union of 
subsets open and closed in $\vp^{-1}(U)$, each mapped isomorphically onto $W$ by $\vp$.
\end{enumerate}
\end{theorem}

There is an analogous semialgebraic version of Theorem \ref{thm:nonemb}. Condition (3)
of the theorem is an analogue for subanalytic (or semialgebraic) sets of the bimeromorphic
(or birational) property of resolution of singularities. The example of a closed half-line in $\IR$
shows that the finite-to-one property in (3) is needed. The fact that $U$ is not required to be 
the entire $k$-dimensional smooth part of $X$ in
Theorem \ref{thm:nonemb} means there is freedom in the construction of the mapping $\vp$
that can be exploited to prove the global smoothing result by essentially local means.

\begin{theorem}[Embedded global smoothing]\label{thm:emb}
Let $V$ be an analytic manifold of dimension $n$, and let $X$ denote a closed subanalytic subset of $V$,
 $\dim X = k$. Then there is a proper analytic mapping $\vp: V' \to V$, where $V'$ is an analytic manifold 
 of dimension $n$, together with a smooth closed analytic subset $X' \subset V'$ of dimension $k$, and
 a simple normal-crossings hypersurface $B' \subset V'$ transverse to $X'$ (i.e., the components of $B'$
 are smooth and simultaneously transverse to $X'$), such that:
 \begin{enumerate}
 \item $\dim\varphi(B') < k$;
 \item $\varphi|_{V'\setminus B'}$ is finite-to-one and of constant rank $n$; 
 \item $\varphi$ induces an isomorphism from a union of components of $X'\setminus B'$ to a smooth
 open subanalytic subset $U$ of $X$ such that $\dim X{\setminus}U <k$.
 \end{enumerate}
 \end{theorem}
 
 The union in (3) is necessarily finite if $X$ is compact; in general, $X$ itself may have infinitely
 many components. The following example shows that the finite-to-one
 property (2) is again needed. In the case that $X$ is a closed semialgebraic subset of $\IR^n$, there 
 is an analogue of Theorem \ref{thm:emb} where the mapping in 
 (2) is one-to-one (see Remark \ref{rem:embsemialg}).
 
 \begin{example}\label{ex:finite}
 Let 
 \begin{equation}\label{eq:sine}
 g(x) := \sin\left(\frac{1}{\de x - 1/\pi}\right),
 \end{equation}
 where $\de > 0$ is a constant. Then $g(x)$ is analytic on the open interval 
 $(-\infty, 1/\de \pi)$. Let $X = \{(x,y)\in \IR^2: y = g(x),\, x \leq 1/2\de\pi\}$. A mapping
 $\vp: V' \to \IR^2$ as in Theorem \ref{thm:emb} must be at least two-to-one on $V'\setminus B'$.
 (Otherwise, the image of $X'$ would provide an extension of $X$ to a closed analytic curve in $\IR^2$.)
 \end{example}
 
 We believe that Theorems \ref{thm:nonemb} and \ref{thm:emb} are not, in general,
 true with the stronger condition that $U$ is the entire smooth part of $X$ of dimension $k$, but we do not have
 a counterexample. The following example in the algebraic case is illustrative.
 
 \begin{example}\label{ex:smoothpart} 
Let $X$ be the algebraic subset of $\IR^3$ defined by $z^4 = x^3 + wxz^2$ (cf.\,\cite[Rmk.\,\allowbreak7.3]{BMihes});
$X$ can be obtained as a blowing-down ($u=x/z$) of the smooth hypersurface $X' \subset \IR^3$ given by
$z = u^3 + uw$. The smooth part of $X$ (as an algebraic \emph{set}) is the complement in $X$ of the half-line
$\{x=z=0,\,w\leq 0\}$. The blowing-up $\varphi: X'\to X$ satisfies Theorem \ref{thm:nonemb} with
$U =$ complement in $X$ of the $w$-axis, but the inverse image of
$\Sing X$ in $X'$ is a ``T-shaped'' set including only the non-positive w-axis.
We can get a mapping as in Theorem \ref{thm:nonemb}, where $U$ is the entire
smooth part of $X$, by following the blowing-up with an additional (generically) $2$-to-$1$ covering.
\end{example}
 
 \subsection{Simplification of an analytic morphism}\label{subsec:equidim}
 Let $\vp: Y \to Z$ denote a proper morphism of \emph{analytic spaces}. We say that $\vp$
 is \emph{finite} if, for every $a\in Y$, the local ring $\cO_{Y,a}$ is a finite $\cO_{Z,\vp(a)}$-module,
 via the ring homomorphism $\vp^*: \cO_{Z,\vp(a)} \to \cO_{Y,a}$. If $\vp$ is finite, then $\vp(Y)$
 is a closed semianalytic subset of $Z$ \cite[Lemma 7.3.6]{HiroPisa}.
 
 Let $\s: Z' \to Z$ denote a morphism given as a composite of blowings-up (more precisely, 
 for every relatively compact open subset $V$ of $Z$, $\s|_{\s^{-1}(V)}:  \s^{-1}(V) \to V$ is the
 composite of a finite sequence of blowings-up over $V$). Given a proper morphism $\vp: Y \to Z$,
 let $\Phi: Y\times_Z Z' \to Z'$ denote the canonical morphism from the fibre-product.
 There is an induced morphism $\vp': Y' \to Z'$, where $Y'$ denotes the smallest closed analytic
 subspace of 
 $Y\times_Z Z'$ containing $Y\times_Z Z'\setminus \Phi^{-1}(B')$, where $B'\subset Z'$ is the exceptional
 divisor of $\s$ (i.e, the critical set of $\s$, in the case that $Z$ is smooth). The morphism
 $\vp'$ is called the \emph{strict transform} of $\vp$.
\[
\begin{tikzcd}
Y' \arrow{rd}{\vp'} \arrow[hookrightarrow]{r} & Y\times_Z Z' \arrow{d}{\Phi} \arrow{r} & Y\arrow{d}{\vp}\\
& Z' \arrow{r}{\s} & Z
\end{tikzcd}
\]

If $\s$ is a blowing-up with centre $C\subset Z$, then $Y'\to Y$ can be identified with
the blowing-up of the pull-back ideal $\vp^*(\cI_C)$ (where $\cI_C\subset\cO_Z$ is
the ideal of $C$, and $\vp^*(\cI_C)\subset \cO_Y$ denotes the coherent ideal generated by
the pull-backs of all local sections of $\cI_C$). This follows essentially from the definitions;
cf. \cite[Chapt.\,4]{HiroPisa}.

\begin{question}\label{qu:equidim}
Given $\vp: Y \to Z$, can we find a composite of blowings-up $\s: Z' \to Z$ such that
$\vp'$ has fibres that are equidimensional in some neighbourhood of every point of $Y'$?
\end{question}

Any closed subanalytic subset $X$ of $Z$ is the image of a proper morphism $\vp: Y\to Z$ with fibres
that generically are finite \cite[Ch.\,7]{HiroPisa}, \cite[Thm.\,0.1]{BMihes}, so a positive answer would
provide a composite of blowings-up $\s$ such that $X' := \vp'(Y')$ is semianalytic (cf. Lemma \ref{lem:flat} below).
In Section \ref{sec:ex} below, we will use the function \eqref{eq:sine} to construct
examples showing that the answer to Question \ref{qu:equidim} is \emph{no}, in general.

\begin{remark}
In the complex-analytic case, the answer is \emph{yes} and, in fact, there is a stronger
result due to Hironaka \cite{HiroAMJ}: $\vp$ can be transformed to a flat morphism by
a composite of blowings-up $\s$. Hironaka's proof is based on successively blowing up
\emph{local flatteners} of the morphism. Remarkably, Hironaka shows that $\vp$ can
be flattened by global blowings-up of $X$ although a global flattener does not exist,
in general, even in the complex case (cf. \cite[Ch.\,4]{HiroPisa}). Equidimensionality of
fibres as a substitute for the stronger flatness condition is studied in \cite{P}.
\end{remark}

As a final remark \ref{rem:final}, we note that a construction similar to that in
Examples \ref{ex:1}, \ref{ex:2} can be used to show that, in the real-analytic category, it is
not true, in general, that a composite of blowings-up is also a blowing-up. It follows
that a characterization of blow-analytic mappings claimed by Fukui \cite[Section 2]{Fuk} 
is not true as stated.

\section{Global smoothing theorems}\label{sec:smoothing} 

\subsection{Lemma of Hironaka}\label{subsec:hiro}
The proofs of our global
smoothing theorems \ref{thm:nonemb} and \ref{thm:emb} use the following local lemma
due essentially to Hironaka \cite[Prop.\,7.3]{HiroPisa} (see also \cite[Thm.\,A.4.1]{BS}).
The lemma is a consequence of Hironaka's local flattening theorem \cite[Ch.\,4]{HiroPisa},
using resolution of singularities to dominate each blowing-up of a local flattener by a
sequence of blowings-up with smooth centres. We recall that a \emph{local blowing-up}
$\s: V'\to V$ means a composite $V' \to W \hookrightarrow V$, where $W \hookrightarrow V$
is the inclusion of an open subset, and $V' \to W$ is a blowing-up.

\begin{lemma}\label{lem:flat}
Let $V$ be an analytic manifold, and let $X$ denote a closed subanalytic subset of $V$. 
Let $K$ be a compact subset of $V$.
Then there exists a finite collection of analytic mappings $\pi_\la: V_\la \to V$, where each 
$V_\la$ is an analytic manifold of dimension $= \dim V$, and a compact subset $K_\la$ of 
$V_\la$, for each $\la$, with the following properties.
\begin{enumerate}
\item $\cup\,\pi_\la(K_\la)$ is a neighbourhood of $X\cap K$ in $V$.
\item For each $\la$, $\pi_\la$ is the composite of a finite sequence of local blowings-up with 
smooth centres. The union of the inverse images of these centres in $V_\la$ is a closed 
analytic hypersurface $B_\la$ of $V_\la$, so that $\pi_\la$ induces an open embedding 
$V_\la \backslash B_\la \to V$. Moreover, 
$\dim \cup\,\pi_\la(B_\la) < \dim X$.
\item (The closure of) $(\pi_\la)^{-1}(X)\backslash B_\la$ is semianalytic, for every $\la$.
\end{enumerate}
\end{lemma}

Let $p$ denote the longest length of the sequence of local blowings-up involved in $\pi_\la$, 
for any $\la$, in Lemma \ref{lem:flat}. We will call $\{\pi_\la\}$ a \emph{semianalytic covering} 
of $X\cap K$ of \emph{depth} $p$. We will prove Theorem \ref{thm:emb} first in the case
that $X$ is semianalytic, and reduce the subanalytic to the semianalytic case by induction 
on the depth of a semianalytic covering, for suitable $K$.

\subsection{Smoothing of a semianalytic $n$-cell}\label{subsec:cell}
Let $V$ be an analytic manifold of dimension $n$, and let $C$ denote the closure of a relatively compact
open semianalytic subset of $V$. We will say that $C$ is a \emph{semianalytic $n$-cell} if there are
finitely many analytic functions
$f_i$, $i=1,\ldots,q$, defined in a neighbourhood $W$ of $C$, such that
$C= \bigcup_{j=1}^r C_j$, 
where each
$$
C_j = \{x\in W: f_i(x) \geq 0,\, i\in I_j\},\, \text{ for some } I_j \subset \{1,\ldots,q\};
$$
in particular, the boundary of $C$,
$\bdry C \subset \bigcup_i \{f_i(x)=0\}$. Note that the \emph{boundary hypersurfaces} $\{f_i(x)=0\}$
may include interior points of $C$.

\begin{lemma}\label{lem:cell}
Let $C$ denote a semianalytic $n$-cell in $V$, as above. Then there is an analytic mapping $\xi: S \to V$,
where $S$ is a compact analytic manifold of dimension $n$, a simple normal crossings 
hypersurface $D\subset S$, and a dense open semianalytic subset $U$ of $C$, such that
$\xi(S) = C$, $S\backslash D = \xi^{-1}(U) =$ a finite union of open and closed subsets
each projecting isomorphically onto $U$.
\end{lemma}

\begin{proof} We can assume that $\dim C_{j_1}\cap C_{j_2} < n$, for all $j_1 \neq j_2$,
and can thus reduce to the case that $C$ is of the form
$$
C = \{x\in W: f_i(x) \geq 0,\,i=1,\ldots,q\}.
$$

Define
$$
Z := \{(x,t) \in W\times\IR^q: t_i^2 = f_i(x),\, i=1,\ldots,q\},
$$
where $t=(t_1,\ldots,t_q)$. Then $Z$ is a compact analytic subset of $W\times\IR^q$. Let $\pi: Z \to W$ denote the
restriction of the projection $W\times\IR^q \to W$. Then $\pi(Z) = C$. 
Moreover, there is a closed analytic subset $Y$ of $Z$, with $\dim Y < n = \dim Z$,
and an open dense semianalytic subset $U$ of $X$, such that $Z{\backslash}Y = \pi^{-1}(U) =$ a finite union
of open and closed subsets, each projecting isomorphically onto $U$. The result then follows
by composing $\pi$ with a mapping $S\to Z$ given by resolution of singularities of $Y\subset Z$; cf.
\cite[Thm.\,5.10]{HiroPisa}, \cite[Thm.\,1.6]{BMinv}.
\end{proof}

\begin{remark}\label{rem:cell} (1) In the case that $X$ is a semalgebraic subset of $\IR^n$, the same
proof gives an analogue of Lemma \ref{lem:cell} where the mapping $\xi$ is algebraic.

\medskip\noindent
(2) Our proof of Theorem \ref{thm:emb} involves Lemma
\ref{lem:cell} for a covering
of $X$ by semianalytic $n$-cells with disjoint interiors. In the case that $X$ is a compact subanalytic
subset of $V=\IR^n$,
Lemma \ref{lem:cell} is needed only in the case of a cube $C \subset \IR^n$ (see Remark \ref{rem:partition}). In this case,
a smoothing $\xi: S \to V$ can be constructed more efficiently as follows. Suppose that
$C=[-1,1]^n$.
Then the projection $(x,y) \mapsto x$ of the unit circle $S^1$ in $\IR^2$ onto the closed interval $[-1,1] \subset \IR$
induces a real-analytic mapping $\xi: S := (S^1)^n \to \IR^n$ onto $C$, such that $\xi$ induces a $2^n$-sheeted
covering of the open cube $(-1,1)^n$, and the inverse image of the boundary is a simple normal crossings
hypersurface $D$ in $S$.
\end{remark}

\subsection{Partition into semianalytic cells}\label{subsec:partition} Let $V$ denote an analytic manifold
(assumed countable at infinity), $\dim V = n$. A locally finite (hence countable) collection of semianalytic
$n$-cells $C_\la$ in $V$ will be called a \emph{partition} of $V$ into semianalytic $n$-cells if $V = \cup\,C_\la$
and the interiors of the $C_\la$ are disjoint. We will say that such a partition $\cP$ is \emph{subordinate} to
a covering $\cC$ of $V$ by open subsets $W_\mu$ if each $C_\la \in \cP$ lies in some $W_\mu$.
We will say that $\cP$ is \emph{compatible} with a semianalytic subset $Y$ of $V$ if the interior $\interior C_\la$
of every $C_\la \in \cP$ lies in either the interior or exterior of $Y$.

Given a subanalytic subset $X$ of $V$, we will say that a partition into semianalytic $n$-cells $C_\la$ is
\emph{in general position} with respect to $X$ if, for each $\la$, the boundary hypersurfaces $\{f_i(x)=0\}$
of $C_\la$ (see \S\ref{subsec:cell}) can be chosen so that $\dim (X\bigcap\{f_i(x)=0\}) < \dim X$, for all $i$.

\begin{lemma}\label{lem:partition} Let $V$ denote an analytic manifold of dimension $n$, and let $\cC$
be an open covering of $V$.
\begin{enumerate}
\item There exists a (locally finite) partition of $V$ into semianalytic $n$-cells, subordinate to $\cC$.
\item If $Y$ is a semianalytic subset of $V$, then there exists a partition of $V$ into semianalytic cells,
subordinate to $\cC$ and compatible with $Y$.
\item If $X$ is a closed subanalytic subset of $V$, there is a partition into semianalytic cells,
subordinate to $\cC$ and in general position with respect to $X$.
\end{enumerate}
\end{lemma}

\begin{proof}
Consider a covering of $V$ by a locally finite (hence countable) collection of analytic coordinate
charts $V_\io$, $\io=1,2,\ldots,$ where each $V_\io$ lies in a member of $\cC$. 
Given $\io$ and a positive integer $q_\io$, let $(x_1,\ldots,x_n)$ 
denote the coordinates of $V_\io$ and consider the
$q_\io$-\emph{grid} of $V$ formed by the hyperplanes $\{x_i = j/q_\io\}$, $j\in\IZ$, $i=1,\ldots,n$.
Let $\wC_{\io\mu}$ denote the closed cubes (of side length $1/q_\io$) determined by the $q_\io$-grid.
Of course, we can choose the covering $\{V_\io\}$ and the $q_\io$ with the property that, for each
$\io$, there is a big closed cube $\wQ_\io \subset V_\io$ with sides determined by the $q_\io$-grid,
such that the interiors $\interior \wQ_\io$ of all $\wQ_\io$ cover $V$; in fact, we can assume that $V$ is covered by smaller
open balls (say, with centre $=$ centre of $\wQ_\io$ and diameter $=$ half the side length of $\wQ_\io$).

Write $Q_1 := \wQ_1$, and $C_{1\mu} := \wC_{1\mu}$, for all $\mu$. For each $\io > 1$, set
\begin{align*}
Q_\io &:= \text{ closure of } \left(\interior \wQ_\io\right) \backslash \bigcap_{\ga < \io} \wQ_\ga,\\
C_{\io\mu} &:= \text{ closure of } \left(\interior \wC_{\io\mu}\right) \backslash \bigcap_{\ga < \io} \wQ_\ga, \text{ for all } \mu.
\end{align*}
Replacing each $q_\io$ by a large enough integral multiple, if necessary, 
we can assume that each $C_{\io\mu}$ is a semianalytic $n$-cell
(in particular, $\bdry C_{\io\mu}$ lies in the union of the zero sets of finitely many analytic functions defined in a
neighbourhood of $C_{\io\mu}$, given by the boundary hypersurfaces of $\wC_{\io\mu}$ and $\wQ_\ga$, $\ga < \io$). 
Then the collection of all cells $C_{\io\mu}$, where
$\wC_{\io\mu} \subset \wQ_\io$, for all $\io$, form a partition of $V$ subordinate to $\{V_\io\}$.
The assertion (1) follows.

Clearly, if $Y$ is a semianalytic subset of $V$, then, after taking a large enough multiple of $q_\io$ above,
each $C_{\io\mu}$ can be partitioned into finitely many cells, each with interior in either the interior or exterior of $Y$,
as required by (2).

Given a closed subanalytic subset $X$ of $V$, $\dim X = k$, we can also assume that, for each $\io$,
every coordinate hyperplane
$\{x_i = j/q_\io\}$ of $V_\io$ intersects $X$ in a subanalytic subset of dimension $< k$ (by a small linear 
coordinate change, if necessary; in fact, it is enough that each hyperplane
$\{x_i = j/q_\io\}$ that intersects $\wQ_\io$ has this property). 
It follows that, for each cell $C_{\io\mu}$ in the resulting
partition, the intersection of
$X$ with every boundary hypersurface $\{f_i(x)=0\}$ has dimension $< k$. This proves (3).
\end{proof}

\begin{remark}\label{rem:partition} The proof of Lemma \ref{lem:partition} shows that, if
$X$ is a compact subanalytic subset of $\IR^n$, then, for any open covering $\cC$ of $\IR^n$,
there is a partition of $\IR^n$ into cubes, subordinate to $\cC$ and in general position with 
respect to $X$.
\end{remark}

\subsection{Proofs of the main theorems}

\begin{proof}[Proof of the embedded global smoothing theorem \ref{thm:emb}]

\noindent
\emph{I. The semianalytic case.} Suppose that $X$ is a closed semianalytic subset of $V$. Then there is a 
locally finite covering of $V$ by open subsets $V_\io$ such that, for each $\io$, there are
closed analytic subsets $Y_\io \subset Z_\io$ of $V_\io$, $\dim Z_i = k$, and an open and closed subset $U_\io$ of 
$Z_\io \backslash Y_\io$, such that $U_\io$ is an open subset of the smooth part of $X$ of dimension $k$ 
and $\dim \left(X\cap V_\io\right){\backslash}U_\io < k$.

By Lemma \ref{lem:partition}, there is a partition $\cP$ of $V$ into semianalytic $n$-cells $C$, subordinate
to $\{V_\io\}$ and in general position with respect to $X$. It is enough to show that, for each $C\in \cP$
such that $C\cap X\neq \emptyset$, there
is a mapping $\vp_C: V'_C \to V$ onto $C$, satisfying the conclusion of the theorem with respect to $X\cap C$. Indeed,
we can then simply let $V'$ be the disjoint union of the $V'_C$ and let $\vp: V' \to V$ be the mapping 
given by $\vp_C$ on each $V'_C$.

Consider such a cell $C$. Choose $\io$ so that $C\subset V_\io$. Take $\xi: S \to V$ 
onto $C$, and $D \subset S$, as in Lemma \ref{lem:cell}. By
resolution of singularities, there exist an analytic manifold $V'_C$ of dimension $n$, a proper analytic surjection $\rho: V'_C \to S$, and
a smooth closed analytic subset $X'$ of $V'_C$ of pure dimension $k$ ($X' =$ strict transform of $Z_\io$), such that
$B' := \rho^{-1}(\xi^{-1}(Y_\io)\cup D)$ is a simple normal crossings
hypersurface in $V_C'$ transverse to $X'$, and $\vp_C := \xi\circ\rho: V'_C \to V$,
together with $X'$ and $B'$, 
satisfy the conclusions of the theorem with respect to $X\cap C$ (see \cite[Thms.\,5.10, 5.11]{HiroPisa},
\cite[Thms.\,1.6, 1.10]{BMinv}).

\medskip\noindent
\emph{II. The general subanalytic case.} Consider a locally finite covering of $V$ by relatively
compact open subsets $V_\io$. By Lemma \ref{lem:flat}, for each $\io$, there is a semianalytic
covering $\{\pi_{\io\la}\}$ of $X\cap \overline{V_\io}$, of depth $p_\io$, say.

Each $\pi_{\io\la}$ is a composite of local blowings-up
$$
\pi_{\io\la} = \pi_{\io\la}^1 \circ \pi_{\io\la}^2 \circ \cdots \circ \pi_{\io\la}^{p(\io,\la)},\quad p(\io,\la)\leq p_\io;
$$
i.e., 
$$
\pi_{\io\la}^i: V_{\io\la}^i \to W_{\io\la}^i \hookrightarrow V_{\io\la}^{i-1},\quad i=1,\ldots,p(\io,\la),
$$
where $W_{\io\la}^i \subset V_{\io\la}^{i-1}$ is an open subset and $\pi_{\io\la}^i: V_{\io\la}^i \to W_{\io\la}^i$ is a
blowing-up with smooth centre ($V_{\io\la}^0 := V$).

By Lemma \ref{lem:partition}, there is a partition $\cP$ of $V$ into semianalytic $n$-cells $C$, subordinate
to $\{V_\io\}$ and in general position with respect to $X$. Let $\cP_X := \{C \in \cP: X\cap C \neq \emptyset\}$.
We can assume that
\begin{enumerate}\item $\cP_X = \bigcup_\io \cP_\io$, where the $\cP_\io$ are disjoint subsets of $\cP_X$
and $Q_\io := \bigcup\{C: C\in\cP_\io\} \subset V_\io$;
\item if $C\in \cP_\io$, then $C\in W_{\io\la}^1$, for some $\la = \la(\io,C)$.
\end{enumerate}
(This is clear, for example, from the construction of $\cP$ in the proof of Lemma \ref{lem:partition}(1),
by taking a large enough multiple of $q_\io$.)

Now, it is enough to prove that, for each $\io$, there is a mapping $\vp_\io: V'_\io \to V$ (onto $Q_\io$) satisfying
the conclusion of the theorem with respect to $X\cap Q_\io$. Fix $\io$. Our proof is by 
induction on the depth $p_\io$ of the semianalytic covering $\{\pi_{\io\la}\}$. The case $p_\io=0$ 
follows from the theorem in the case that $X$ is semianalytic. 

Again, it is enough to prove that, for each $C \in \cP_\io$, there is a mapping $\vp_C: V'_C \to V$ (onto $C$)
satisfying the conclusion of the theorem with respect to $X\cap C$. Fix $C \in \cP_\io$.
Let $B^1$ denote the exceptional divisor of $\pi_{\io\la}^1$, where $\la = \la(\io,C)$,
and let $X^1$ denote the closure in $V_{\io\la}^1$ of $(\pi_{\io\la}^1)^{-1}(X\cap C)\backslash B^1$.

Then $X^1 \subset V_{\io\la}^1$ has a semianalytic covering of depth $< p_\io$.

By induction, there is a proper analytic mapping $\psi: T \to V_{\io\la}^1$, where $T$ is an analytic
manifold of dimension $n$, together with a smooth closed analytic subset $Z$ of $T$, $\dim Z = k$, and a 
simple normal crossings hypersurface $E \subset T$ transverse to $Z$, satisfying the conclusions of the
theorem with respect to $X^1 \subset V_{\io\la}^1$. In particular, $\psi$ induces an isomorphism of a union
of components of $Z\backslash E$ with a smooth open subanalytic subset of $X^1$ whose complement
in $X^1$ has dimension $< k$.

Set $\eta := \pi_{\io\la}^1\circ \psi: T \to V$. Let $\xi: S \to V$ denote an analytic mapping onto $C$,
with a simple normal crossings hypersurface $D \subset S$, satisfying Lemma \ref{lem:cell}.
Consider the fibre product $S \times_V T$ of
$\xi: S \to V$ and $\eta$, and let $\pi_S,\,\pi_T$ denote the projections of $S \times_V T$ to $S,\,T$, respectively.
By resolution of singularities, there is a surjective analytic mapping
$\rho: V_C' \to S \times_V T$, where $V_C'$ is a compact analytic manifold of dimension $n$, such that
the strict transform $X'_C$ of $\pi_T^{-1}(Z)$ is smooth, and the union in $V_C'$ of the
inverse images of $B^1,\, D$ and $E$ is 
a simple normal crossings hypersurface $B'_C$ transverse to $X'_C$. Then the
mapping $\vp_C: V_C' \to V$ given by $\rho$ followed by the projection to $V$ satisfies the conclusions of
the theorem with respect to $X\cap C$, as required. 
\end{proof}

\begin{remark}\label{rem:embsemialg}
In the case that $X$ is a closed semialgebraic subset of $V=\IR^n$, there are global closed algebraic subsets
$Y \subset Z$ of $\IR^n$, where $Z \backslash Y$ is smooth, and an open and closed subset $U$ of 
$Z \backslash Y$, such that $U$ is an open subset of the smooth part of $X$ of dimension $k$,
and $\dim X{\backslash}U < k$ (cf. case I of the proof above). By resolution of singularities, there is a
sequence of blowings-up with smooth algebraic centres over $Y$, after which the strict transform $Z'$ of $Z$
is smooth, and the inverse image of $Y$ is a simple normal crossings hypersurface transverse to $Z'$.
We thus get a semialgebraic analogue of Theorem \ref{thm:emb}, where the mapping in condition (2) of the 
theorem is one-to-one.
\end{remark}

\begin{proof}[Proof of the non-embedded smoothing theorem \ref{thm:nonemb}]
By Theorem \ref{thm:emb}, we can assume that $X$ is the closure of an open semianalytic 
subset of $V$. By Lemma \ref{lem:partition}(2), there is a partition of $V$ into semianalytic $n$-cells
compatible with $X$. In particular, $X$ is a locally finite union of semianalytic $n$-cells, so the result
follows from the special case that $X$ is itself a semianalytic $n$-cell---this is the result of Lemma \ref{lem:cell}.
\end{proof}

The semialgebraic version of Theorem \ref{thm:nonemb} can be proved in the same way (see Remark \ref{rem:cell}(1)).

\section{Examples}\label{sec:ex}
We begin with two examples of a proper analytic mapping $\vp: V \to W$, where $V$ is an analytic space
of dimension 3 and $W=\IR^3$, with the property that there is no mapping $\s: W' \to W$ given as the composite of a sequence
of global blowings-up, such that the strict transform $\vp'$ of $\vp$ by $\s$ has all fibres finite (or empty). 
Each of the examples below involves the function \eqref{eq:sine}, where $\de> 0$ is small.  

\begin{example}\label{ex:1}
Let
$S^3 := \{(x,y,z,w)\in \IR^4: x^2+y^2+z^2+w^2=1\}$, and let $C := \{(x,y,z,w)\in S^3: z=0,\, y=g(x)\}$.
If $\de> 0$ is small enough (e.g., $\de \leq 1/3\pi$), then $C$ is a smooth curve.
We define $\vp: V\to W$ as the composite
$$
V \xrightarrow{\pi_C} S^3 \xrightarrow{p} \IR^3 \xhookrightarrow{\io} Z \xrightarrow{\pi_0} W=\IR^3,
$$
where 
\begin{itemize}
\item[$\pi_0$] is the blowing-up of $0\in \IR^3$;
\item[$\io$] is the inclusion of the $z$-coordinate chart, so that $\pi_0\circ\io: (x,y,z) \mapsto (xz, yz, z)$, and
$\{z=0\}$ represents the exceptional divisor $D$ of $\pi_0$ in this chart;
\item[$p$] is induced by the projection $(x,y,z,w) \mapsto (x,y,z)$;
\item[$\pi_C$] is the blowing-up with centre $C$ (so $\pi_C$ has $1$-dimensional fibres over $C$).
\end{itemize}
Note that $p(C) \subset D$. The required property of $\vp$ is a consequence of the fact that $p(C)$ does not extend to a 
closed analytic curve in $Z$.

Indeed, suppose there is a composite 
of global blowings-up, $\s: W' \to W$, such that the strict transform $\vp': V'\to V$ of $\vp$ by $\s$ has all fibres finite.   
Say $\s=\s_1\circ\s_2\circ\cdots\circ\s_k$, where each $\s_j: W_j \to W_{j-1}$ is a blowing-up with smooth centre
$C_{j-1}\subset W_{j-1}$ ($W_0=W$, $W_k=W'$). Then there is a commutative diagram,
\[
\begin{tikzcd}
Z' = \hspace{-3.5em} & Z_k\arrow{d}{\pi_k}\arrow{r}{\s'_k} & Z_{k-1}\arrow{d}{\pi_{k-1}}\arrow{r} &\cdots \arrow{r} & Z_1\arrow{d}{\pi_1}\arrow{r}{\s'_1} & Z_0\arrow{d}{\pi_0} &\hspace{-3.35em}= Z\\
W' = \hspace{-3em} & W_k\arrow{r}{\s_k} & W_{k-1}\arrow{r} &\cdots \arrow{r} & W_1\arrow{r}{\s_1} & W_0 &\hspace{-2.75em}= W
\end{tikzcd}
\]
where each $\s_j'$ is a composite of finitely many blowings-up with smooth centres: This can be proved inductively.
Given $\pi_j: Z_j \to W_j$, let $T_{j+1}\to W_{j+1}$ be the strict transform of $\pi_j$ by $\s_{j+1}$, and let 
$\tau_{j+1}: T_{j+1}\to Z_j$ denote the associated mapping; i.e., $\tau_{j+1}$ is the blowing-up of the
pull-back ideal $\pi_j^*(\cI_{C_j}) \subset \cO_{Z_j}$, where $\cI_{C_j} \subset \cO_{W_j}$ is the ideal of $C_j$
(see \S\ref{subsec:equidim}). By resolution of
singularities, $\tau_{j+1}$ can be dominated by a finite sequence of blowings-up with smooth centres. More precisely, there is a composite 
$\s_{j+1}': Z_{j+1} \to Z_j$ of finitely many blowings-up with smooth centres, which principalizes $\pi_j^*(\cI_{C_j})$, and $\s_{j+1}'$
factors through $T_{j+1}$, by the universal mapping property of the blowing-up $\tau_{j+1}$. So we get $\pi_{j+1}: Z_{j+1} \to T_{j+1} \to W_{j+1}$.

Let $\s' := \s'_1\circ\s'_2\circ\cdots\circ\s'_k:  Z' \to Z$.
Write $\psi := \io\circ p\circ\pi_C$, and  let $\psi': V''\to Z'$ denote the strict transform of $\psi$ by $\s'$.
Since $\vp'$ has all fibres finite, it follows that $\psi'$ has all fibres finite. 
Indeed, by definition, $V'$ and $V''$ are closed subspaces of the fibre products $V\times_W W'$ and
$V\times_Z Z' \subset (V\times_W W')\times _{W'} Z'$, respectively, and moreover, $V'' \subset 
V'\times_{W'}  Z'$.  This means that each fibre of $\psi'$ is a subset of a fibre of 
$\vp'$.

For each $j=0,\ldots,k-1$, let $C_j'\subset Z_j'$ denote the smallest closed analytic subset
containing $\pi_j^{-1}(C_j)\backslash \pi_j^{-1}(D_j)$, where $D_j$ denotes the exceptional divisor of $\s_j$ ($D_0 = D$). Then
$\dim C_j' \leq 1$ ($C_j'$ may be empty.).
The curve 
$\Ga = \{(x,y,z)\in \IR^3: z=0,\,y = g(x),\, x < 1/\de\pi\} \subset Z$ cannot lie entirely in $C_0'$; therefore, it lifts
to a unique curve $\Ga_1 \subset Z_1$. Likewise, $\Ga_1$ does not lie in $C_1'$, etc. 
(Here we use the property that every subanalytic set containing $\Ga$ is of dimension $\geq 2$; clearly, this property 
is inherited by $\Ga_1$, etc.)  Finally, $\Ga$ lifts to 
a unique curve $\Ga' \subset Z'$, and $\Ga'$ intersects the union of the inverse images of all $\pi_j^{-1}(D_j)$ in a discrete
set. Therefore,
$\psi'$ has one-dimensional fibres over the lifting of a non-empty open subset of 
$p(C)$; a contradiction.
\end{example}

\begin{remark}\label{rem:lift} 
In general, consider a proper analytic  mapping $\vp: V \to W$ which factors through a blowing-up 
$\pi : Z\to W$ of a coherent ideal sheaf in $\cO_W$; i.e., $\vp = \pi \circ \psi$, where $\psi : V\to Z$.  Suppose there is
a composite $\s$ of global blowings-up over $W$ with smooth centres, such that the strict transform 
of $\vp$ by $\s$ has all fibres finite (or empty). Then, by the argument in Example \ref{ex:1}, the strict transform of 
$\psi$ by a composite of such blowings-up over $Z$ also has all fibres finite.
\end{remark}

In Example \ref{ex:1}, we can replace $W$ by an arbitrarily small open ball in $\IR^3$ centred at the origin, and restricting $\vp$ over such a ball will not change the preceding property. It is true,
however, that $\vp$ can be transformed to a morphism with all fibres finite by blowing up at each step with
centre that is globally defined in some neighbourhood of the image of the corresponding morphism 
(e.g., after the first blowing-up
$\pi_0$, with centre globally defined in a neighbourhood of  the image of $\psi$ containing $p(C)$).
The latter phenomenon does not occur in the following example.

\begin{example}  \label{ex:2}
Let $\IR^4 = \IR^3\times\IR \to  \IR^3$
denote the projection $(x,y,z,w) \mapsto (x,y,z)$, and let $S\subset \IR^4$ denote the algebraic subset defined by
\begin{equation}\label{eq:ex1}
(x^2+z^2)^2(w^4+z^2w^2) - (x^2-z^2)^2 = 0.
\end{equation}
Then $S$ is irreducible, $\{x=z=0\} \subset S$, and $S$ maps onto $\IR^3$ since we can solve \eqref{eq:ex1} for
$w^2$ when $x^2+z^2 \neq 0$. (The closure of $S{\backslash}\{x=z=0\}$ maps properly onto $\IR^3$.) 

Let $\pi_0 : M \to \IR^3$ denote  the blowing-up of the origin $\{x=y=z=0\} \subset \IR^3$. 
Then there is a commutative diagram
\[
\begin{tikzcd}
S' \arrow{d} \arrow[hookrightarrow]{r} & M\times\IR \arrow{d}{\pi_0\times\id}\arrow{r} & M \arrow{d}{\pi_0}\\
S \arrow[hookrightarrow]{r} & \IR^3\times\IR \arrow{r} & \IR^3
\end{tikzcd}
\]
where $S'$ denotes the strict transform of $S$ by the blowing-up $\pi_0\times\id$ of $\IR^3\times\IR$. 
Let $S' \to \IR^3$ denote the induced mapping.

Let $U_z$ denote the $z$-coordinate chart of $\pi_0$, i.e., the chart with coordinates $(X,Y,Z)$ in which $\pi_0$ is given by 
$(x,y,z) = (XZ,YZ,Z)$. The mapping $U_z\times\IR \to \IR^3$ given by the diagram above is $(X,Y,Z,w) \mapsto (XZ,YZ,Z)$, and $S'$ is defined in $U_z\times\IR$ by the equation 
\[
(X^2+1)^2(w^4 + Z^2w^2) - (X^2-1)^2 =0.
\]

Setting $Z=0$, this equation splits as
\[
((X^2+1)w^2 - (X^2-1))((X^2+1)w^2 + (X^2-1)) = 0.
\]
Let $C\subset S'$ denote the compact smooth curve defined by
\[
Z=0,\quad Y=g(X),\quad (X^2+1)w^2 + (X^2-1) = 0,
\]
and let $\pi_C: V \to S'$ denote the blowing-up of $S'$ with centre $C$. Then
the mapping 
\[
\vp: V \xrightarrow{\pi_C} S' \longrightarrow \IR^3 = W
\]
has the required property.  

Indeed, suppose that that strict transform of $\vp$ by the composite of a sequence of global blowings-up
over $W$ has all fibres finite. Let $\psi :V\to M$ be the mapping such that $\vp = \pi_0 \circ \psi$.  
By Remark \ref{rem:lift}, there is 
a composite $\s': M'\to M$ of global blowings-up, such that the strict transform $\psi'$ of $\psi$ by 
$\s'$ has all fibres finite.  Then the curve $\Ga = \{(X,Y,Z)\in \IR^3: Z=0,\, Y= g(X),\, X < 1/\de\pi\} \subset M $  
can be lifted to $M'$, and this leads to a contradiction by the same argument as in Example \ref{ex:1}.
\end{example}  

\begin{remark}\label{rem:final} A construction similar to that in the examples above
can be used to show that, in the real-analytic category, it is not necessarily true that
a composite of blowings-up is also a blowing-up. For example, let $\pi_0: Z_1 \to \IR^4$
be the blowing-up of the origin, and let $H$ denote a projective hyperplane in the exceptional
divisor of $\pi_0$. Let $\pi_H: Z_2 \to Z_1$ denote the blowing-up with centre $H$. Consider
an affine coordinate chart $U_1$ of $Z_1$ with coordinates $(x,y,Z,W)$, where $\{W=0\}$
is the exceptional divisor of $\pi_0$ and $H = \{Z=W=0\}$. Let $U_2$ denote the affine
chart of $Z_2$ over $U_1$ with coordinates $(x,y,z,w)$ such that $\pi_H$ is given on $U_2$ by
$(x,y,Z,W) = (x,y,zw,w)$. Let $\pi_C: V \to Z_2$ denote the blowing up with centre
$$
C = \{(x,y,z,w) \in U_2: w=0,\, x^2+y^2+z^2=1,\, y=g(x)\},
$$
and set $\vp := \pi_0\circ\pi_H\circ\pi_C$.

We claim that $\vp$ is not the blowing-up of an ideal. Suppose that $\vp$ is the blowing-up
of an ideal $\cI\subset\cO_{\IR^4}$. Then $\psi = \pi_H\circ\pi_C$ is the blowing-up of
$\pi_0^*\cI$; therefore, the set $A = \{b \in Z_1: \dim \psi^{-1}(b) \geq 2\}$ lies in a real analytic curve
(since, for example, $\psi$ admits a proper complexification).
But this is impossible because $A$ contains a non-empty open subset of $\{Z=W=0,\, y=g(x)\}$.  
\end{remark}

\bibliographystyle{amsplain}

\end{document}